\documentclass[11pt]{amsart}
\usepackage{mathtools}
\usepackage{amsmath}
\usepackage{amssymb}
\usepackage{amsthm}
\usepackage{yhmath}
\usepackage{graphicx}
\usepackage{mathrsfs}
\usepackage{bbm}
\usepackage{xcolor}
\usepackage{tikz-cd}
\usepackage{cleveref}
\usepackage{verbatim}
\usepackage{enumitem}
\DeclareMathAlphabet{\mathpzc}{OT1}{pzc}{m}{it}

\newtheorem{theorem}{Theorem}[section]

\newtheorem*{theorem*}{Theorem}
\newtheorem*{definition*}{Definition}

\newtheorem*{prop*}{Proposition}
\newtheorem*{cor*}{Corollary}

\newtheorem*{claim*}{Claim}

\newtheorem{lemma}[theorem]{Lemma}
\newtheorem{lem}[theorem]{Lemma}

\newtheorem{cor}[theorem]{Corollary}

\newtheorem{prop}[theorem]{Proposition}

\newtheorem{thm}[theorem]{Theorem}

\theoremstyle{definition}
\newtheorem{definition}[theorem]{Definition}
\newtheorem{Def}[theorem]{Definition}

\theoremstyle{remark}
\newtheorem{remark}[theorem]{Remark}
\newtheorem{Rmk}[theorem]{Remarks}
\newtheorem{Q}[theorem]{Question}

\numberwithin{equation}{section}
\newcommand{\F}{\cal F}

\newcommand{\op}{\operatorname}

\newcommand{\Ga}{\Gamma}
\newcommand{\ga}{\gamma}

\newcommand{\la}{\lambda}
\newcommand{\La}{\Lambda}
\newcommand{\ba}{\backslash}

\newcommand{\cal}{\mathcal}
\newcommand{\br}{\mathbb R}

\newcommand{\bH}{\mathbb H}

\newcommand{\G}{\Gamma}

\newcommand{\Furst}{\cal{F}}
\renewcommand{\frak}{\mathfrak}
\newcommand{\e}{\epsilon}

\newcommand{\be}{\begin{equation}}
	\newcommand{\ee}{\end{equation}}

\newcommand{\inte}{\op{int}}
\renewcommand{\L}{\mathcal L}
\newcommand{\fa}{\mathfrak a}

\newcommand{\fg}{\mathfrak{g}}
\newcommand{\fp}{\mathfrak{p}}
\renewcommand{\i}{\op{i}}

\newcommand{\fk}{\mathfrak{k}}

\newcommand{\z}{\mathbb Z}

\renewcommand{\e}{\varepsilon}
\renewcommand{\epsilon}{\varepsilon}
\newcommand{\pc}{P^{\circ}}

\newcommand{\E}{\mathcal E}
\newcommand{\Om}{\Omega}

\newcommand{\R}{\mathbb{R}}

\newcommand{\PSL}{\op{PSL}}

\begin{document}
	
	\graphicspath{ {} }
	
	\pagestyle{plain}
	\title{On denseness of horospheres in higher rank homogeneous spaces}
	\author{Or Landesberg}
	\address{Department of Mathematics, Yale University, New Haven, CT 06520}
	\email{or.landesberg@yale.edu}

\author{Hee Oh}
\address{Department of Mathematics, Yale University, New Haven, CT 06520 and Korea Institute for Advanced Study, Seoul, Korea}
\email{hee.oh@yale.edu}
\thanks{Oh is partially supported by NSF grant DMS-1900101}

	\begin{abstract}
		Let $ G $ be a connected semisimple real algebraic group and $\Ga<G$ be a Zariski dense discrete subgroup.
		Let $N$ denote  a maximal horospherical subgroup of $G$, and $P=MAN$ the minimal parabolic subgroup which is the normalizer of $N$.
		Let $\cal E$ denote the unique $P$-minimal subset of $\Ga\ba G$ and let $\E_0$ be a $P^\circ$-minimal subset.
		We consider a notion of a horospherical limit point in the Furstenberg boundary $ G/P $ and show that the following are equivalent for any $[g]\in \cal E_0$,
		\begin{enumerate}
		    \item $gP\in G/P$ is a horospherical limit point;
			\item $[g]NM$ is dense in $\cal E$;
	 		\item $[g]N$ is dense in $\cal E_0$.
		\end{enumerate}
		The equivalence of (1) and (2) is due to Dal'bo in the rank one case. We also show that unlike convex cocompact groups of rank one Lie groups, the $NM$-minimality of $\E$ does not hold in a general Anosov homogeneous space.
		\end{abstract}
	\maketitle

\section{Introduction}

Let $G$ be a connected semisimple real algebraic group. 
Let $(X, d)$ denote the associated Riemannian symmetric space. 
Let $P=MAN$ be a minimal parabolic subgroup of $G$ with fixed Langlands decomposition, where $A$ is a maximal real split torus of $G$, $M$ the maximal compact subgroup of $P$ commuting with $A$  and $N$ the unipotent radical of $P$. Note that $N$ is a maximal horospherical subgroup of $G$, which is unique up to conjugations.

Fix a positive Weyl chamber 
$\fa^+\subset \log A$ so that $\log N$ consists of positive root subspaces, and we set $A^+=\exp \fa^+$. This means that $N$ is a contracting horospherical subgroup in the sense that
for any $a$ in the interior of $A^+$,
$$N =\{g\in G: a^{-n} g a^n\to e\text{ as $n\to +\infty$}\}.$$

Let $\Ga$ be a {\it Zariski dense} discrete subgroup of $G$.
In this paper, we are interested in the topological behavior of the action of the
horospherical subgroup $N$ on $\Ga\ba G$ via the right translations. When $\Ga<G$ is a cocompact lattice,
every $N$-orbit is dense in $\Ga\ba G$, i.e., the $N$-action on $\Ga\ba G$ is minimal. This is due to Hedlund \cite{Hedlund} for $G=\PSL_2(\br)$ and
 to Veech \cite{Vee} in general.  Dani gave a full classification of possible orbit closures of $N$-action for any lattice $\Ga<G$ \cite{Dani}.

For a general discrete subgroup $\Ga<G$, the quotient space $\Ga\ba G$ does not necessarily admit a dense $N$-orbit, even a dense $NM$-orbit, for instance in the case where $\Ga$ does not have a full limit set. Let $\F$ denote the Furstenberg boundary $G/P$. 
We denote by $\La=\La_\Ga$ the limit set of $\Ga$, 
$$\La=\{\lim_{i\to \infty} \ga_i(o) \in \cal F : \ga_i\in \Ga\}$$
where $o\in X$ and the convergence is understood as in Definition \ref{conver}. This definition is independent of the choice of $o\in X$.
The limit set $\La$ is known to be the unique $\Ga$-minimal subset of $\F$ (see \cite{Benoist,Guivarch,LO1}).
Thus the set
$$\cal E=\{[g]\in \Ga\ba G: gP\in \La\}$$
is the unique $P$-minimal subset of $\Ga\ba G$.
For a given point $[g]\in \cal E$,
the topological behavior of the horospherical orbit $[g]N$ (or of $[g]NM$) is closely related to the ways in which the orbit $\Ga(o)$
approaches $gP$ along its limit cone. The limit cone $\L=\L_\Ga$ of $\Ga$ is defined as the smallest closed cone of $\fa^+$ containing the Jordan projection $\lambda(\Ga)$. It is a convex cone with non-empty interior: $\inte \L\ne \emptyset$ \cite{Benoist}.
If $\op{rank } G=1$, then $\L=\fa^+$. In higher ranks, the limit cone of $\G$ depends more subtly on $\Ga$.

 \subsection*{Horospherical limit points} Recall that in the rank one case, a horoball in $X$ based at $\xi\in \F$ is
 a subset of the form $gN (\exp \fa^+ )( o)$
 where $g\in G$ is such that $\xi=gP$ \cite{Dal}.
Our generalization to higher rank of the notion of a horospherical limit point involves the limit cone of $\Gamma$.  By a $\Gamma$-tight horoball based at $\xi\in \F$, we mean a subset of the form $\cal H_{\xi}=g N (\exp \cal C)  (o)$ where $g\in G$ is such that $\xi=gP$ and $\cal C$ is a closed cone contained in $\inte \L\cup\{0\}$.
		For $T>0$, we write 
		\[\cal H_{\xi}(T)= gN (\exp (\cal C-\cal C_T))o\] 
		where $\cal C_T=\{u\in \cal C: \|u\|<T\}$ for a Euclidean norm $\|\cdot\|$ on $\fa$.

      \begin{Def} We call a limit point $\xi\in \Lambda$ a horospherical limit point of $\Ga$ if 
      one of the following equivalent conditions holds:
      \begin{itemize}[leftmargin=*]
          \item 
   there exists a $\Ga$-tight horoball $\cal H_\xi$ based at $\xi$ such that for any $T>1$, $\cal H_\xi(T)$ contains some point of $\Gamma (o)$;
   \item  there exist a closed cone $\cal C \subset {\inte \L} \cup \{0\}$ and a sequence $ \gamma_j \in \Gamma $ satisfying that
	 $\beta_\xi(o,\gamma_jo) \in \cal C$  for all $j\ge 1$ and
  $ \beta_\xi(o,\gamma_jo) \to \infty $ as $j\to \infty$,
 where $\beta$ denotes the $\fa$-valued Busemann map (Definition \ref{bus}).    \end{itemize}
\end{Def}

\noindent See  Lemma \ref{hd} for the equivalence of the above two conditions.
We denote by $$\La_h\subset \La$$ the set of all horospherical limit points of $\Ga$. The attracting fixed point $y_\ga$ of a loxodromic element $\ga\in \Ga$ whose Jordan projection $\lambda(\gamma)$ belongs to $\inte \L$ is always a horospherical limit point (Lemma \ref{hor}). Moreover, for any $u\in \inte \L$, any $u$-directional radial limit point $\xi$
(i.e, $\xi=gP$ for some $g\in G$ such that $\limsup_{t\to \infty} \Gamma g \exp (tu)\ne \emptyset$)  is also a horospherical limit point (Lemma \ref{st}).

\begin{Rmk}\;
	\begin{enumerate}
		\item There exists a notion of horospherical limit points in the geometric boundary associated to a symmetric space, see \cite{Hattori}. When $\op{rank }G \geq 2 $, this notion and the one considered here are different.
		\item Unlike the rank one case, a sequence $\gamma_i(o)\in \cal H_{\xi}(T_i)$, with $T_i\to\infty$, does not necessarily {\it converge} to $\xi$ for a $\Gamma$-tight horoball $\cal H_\xi$ based at $\xi$. It is hence plausible that a general discrete group $ \Gamma $ would support a horospherical limit point outside of its limit set. 
	\end{enumerate}
\end{Rmk}

\subsection*{Denseness of horospheres} The following theorem generalizes Dal'bo's theorem \cite{Dal} to discrete subgroups in higher rank semisimple Lie groups:
\begin{thm}\label{m1} Let $\Ga<G$ be a Zariski dense discrete
subgroup. For any $[g]\in \cal E$, the following are equivalent:
\begin{enumerate}
\item $gP\in \La_h$;
 \item $[g] NM$ is dense in $\cal E$. 
 \end{enumerate}
\end{thm}

\begin{Rmk}
  Conze and Guivarc'h  considered the notion of  a horospherical limit point for Zariski dense discrete subgroups  $\Gamma$ of $\op{SL}_d(\br)$ using the description of $\op{SL}_d(\br)/P$ as the full flag variety and the standard linear action of $\Ga$ on $\br^d$ 
  \cite{Conze-Guivarch}. By duality, this notion coincides with ours and hence the special case of Theorem \ref{m1} for $G=\op{SL}_d(\br)$ also follows from \cite[Theorem 4.2]{Conze-Guivarch}\footnote{We remark that in \cite[Theorem 6.3]{Conze-Guivarch} it is moreover claimed that for every Zariski-dense Schottky group $\Gamma < \op{SL}_d(\R)$, the $NM$-action on $\cal E$ is minimal. We were unable to verify this claim as its proof seems to contain a significant, possibly unfixable, gap.}.
\end{Rmk}

In order to extend Theorem \ref{m1} to $N$-orbits, we  fix a $P^\circ$-minimal subset $\cal E_0$ of $\Ga\ba G$ where
$P^\circ$ denotes the identity component of $P$. Clearly, $\cal E_0\subset \cal E$.
Since $P=P^\circ M$,  any $P^\circ$-minimal subset is a translate of $\cal E_0$ by an element of the finite group $M^\circ \backslash M$, where $ M^\circ $ is the identity component of $ M $. Denote by $ \frak D_\Gamma = \{\cal E_0,..., \cal E_p\} $ the finite collection of all $ P^\circ $-minimal sets in $ \cal E $. 
In order to understand $N$-orbit closures it is hence sufficient to restrict to $\cal E_0$.

The following is a refinement of Theorem \ref{m1}: \begin{thm}\label{m2}
Let $\Ga<G$ be a Zariski dense discrete
subgroup. For any $[g]\in \cal E_0$, the following are equivalent:
\begin{enumerate}
\item $gP\in \La_h$;
 \item $[g] N$ is dense in $\cal E_0$. 
 \end{enumerate}
 \end{thm}

\begin{remark}
	We may consider horospherical limit points outside the context of $\La$. In this case our proofs of Theorems \ref{m1} and \ref{m2} show that if $gP\in \cal F$ is a horospherical limit point,
then the closures of $[g]MN$ and $[g]N$ contain $\cal E$ and $\cal E_i $, for some $ \cal E_i \in \frak D_\Gamma $, respectively.

\end{remark}

For $G=\op{SO}^\circ({n,1})$, $n\ge 2$, Theorem \ref{m2} was proved in \cite{MS}.
When $G$ has rank one  and $\Ga<G$ is convex cocompact, 
every limit point is horospherical and Winter's mixing theorem \cite{Wi}
implies the $N$-minimality of $\cal E_0$. 
\medskip

\subsection*{Directional horospherical limit points} 
We also consider the following seemingly much stronger notion: \begin{definition} For $u\in \fa^+$,
	 a point $ \xi \in \cal F$ is called $ u $-horospherical 
	if  there exists  a sequence $ \gamma_j \in \Gamma $ such that
 $\sup_j \|  \beta_\xi(o,\gamma_jo)- \br_+ u\|<\infty$ and
 $ \beta_\xi(o,\gamma_jo)\to \infty$ as $j\to \infty$.
	\end{definition}
 Denote by $ \Lambda_h(u) $ the set of $ u $-horospherical limit points. Surprisingly, it turns out that  every horospherical limit point is $u$-horospherical {\it{for all}} $u\in \inte\L$:
\begin{theorem}
    For all $u\in \inte \L$, we have $$\Lambda_h=\Lambda_h(u).$$
\end{theorem}

\subsection*{Existence of non-dense horospheres}
A finitely generated subgroup $\Ga<G$ is called an Anosov subgroup  (with respect to $P$) if there exists $C>0$ such that
for all $\ga\in \Ga$, $\alpha(\mu(\ga))\ge C|\ga| -C$ for all simple roots
$\alpha$ of $(\frak g, \fa^+)$,
where $\mu(\ga)\in \fa^+$ denotes the Cartan projection of $\ga$ and $|\ga|$
is the word length of $\ga$ with respect to a fixed finite generating set of $\Ga$. 

For Zariski dense Anosov subgroups of $G$, almost all $NM$-orbits are dense in $\cal E$ 
and almost all $N$-orbits are dense in $\cal E_0$ with respect to any Patterson-Sullivan measure on $\La$ (\cite{LO1}, \cite{LO2}). In particular,
the set of all horospherical limit points has full Patterson-Sullivan measures.

On the other hand, as Anosov subgroups are regarded as higher rank generalizations of convex cocompact subgroups, it is a natural question whether the minimality of the $NM$-action persists in the higher rank setting. It turns out that it is not the case. Our example is based on Thurston's theorem \cite[Theorem 10.7]{Th} together with the following observation on the implication of the existence of a Jordan projection of an element of $\Ga$ lying in the boundary $\partial \L$ of the limit cone.

 \begin{prop}\label{thp}  Let $\Ga<G$ be a Zariski dense discrete subgroup. For any loxodromic element $\ga\in \Ga$,
 we have
 $$\text{  $\lambda(\gamma)\in  \inte \L \quad $  if and only if  $\quad \{y_{\ga}, y_{\ga^{-1}}\} \subset \La_h $}$$
 where $y_\ga$ and $y_{\ga^{-1}}$ denote the attracting fixed points of $\ga$ and $\ga^{-1}$ respectively. 

In particular, if $\lambda(\Ga)\cap \partial \L\ne \emptyset$,
     then $\La\ne \La_h$ and hence there exists a non-dense $NM$-orbit in $\cal E$.
    \end{prop}
     
Thurston's work \cite{Th} provides many examples of Anosov subgroups satisfying that $\lambda(\Ga)\cap \partial \L\ne \emptyset$. To describe them,
let $\Sigma$ be a a torsion-free cocompact lattice of $\PSL_2(\br)$ and let $\pi:\Sigma\to \PSL_2(\br)$ be a discrete faithful representation. Let $0<d_-(\pi) \le d_+ (\pi) <\infty $ be the minimal and maximal geodesic stretching constants:
 \begin{equation}\label{dd} d_+(\pi) =\sup_{\sigma\in \Sigma-\{e\} }\tfrac {\ell(\pi(\sigma))}{\ell(\sigma)} \quad \mbox{and} \quad d_-(\pi) =\inf_{\sigma\in \Sigma-\{e\}}\tfrac{\ell(\pi(\sigma))}{\ell(\sigma)}
 \end{equation} 
where
$\ell(\sigma)$ denotes the length of the closed geodesic in the hyperbolic manifold $\Sigma\ba \bH^2$ corresponding to $\sigma$
and $\ell (\pi(\sigma))$ is defined similarly.

Consider the following self-joining subgroup $$\Ga_\pi:=(\text{id}\times \pi)(\Sigma)=\{(\sigma, \pi(\sigma)):\sigma\in \Sigma\} <\PSL_2(\br) \times \PSL_2(\br).$$
It is easy to see that $\Ga$ is an Anosov subgroup of $G=\PSL_2(\br) \times \PSL_2(\br)$. Moreover when $\pi$ is
not a conjugate by a M\"obius tranformation,
$\Ga_\pi$ is Zariski dense in $G$ (cf. \cite[Lemma 4.1]{KO}). Identifying $\fa=\br^2$, 
the Jordan projection $\lambda(\gamma_\pi)$ of  $\gamma_\pi=(\sigma, \pi(\sigma))\in \Ga_\pi$
is given by $(\ell(\sigma), \ell(\pi(\sigma)))\in \br^2$. Hence
the limit cone $\L$ of $\Gamma_\pi$ is given by
$$\cal L:=\{(v_1,v_2)\in \br_{\ge 0}^2: d_-(\pi) v_1\le v_2\le d_+(\pi) v_1\}. $$

Thurston \cite[Theorem 10.7]{Th} showed
    that  $d_+(\pi)$ is realized by
    a simple closed geodesic of $\Sigma\ba \bH^2$ in {\it most of cases}, which hence provides infinitely many examples of $\Ga_\pi$ which satisfy $\lambda(\Ga_\pi)\cap \partial \L\ne \emptyset$.
Therefore Proposition \ref{thp} implies (in this case, we have $NM=N)$:
\begin{cor}\label{th2} There are infinitely many non-conjuagte Zariski dense Anosov subgroups $\Ga_{\pi}<\PSL_2(\br) \times \PSL_2(\br)$ with non-dense $NM$-orbits in $\cal E$.
    \end{cor}

We close the introduction by the following question (cf.~\cite{Kim},\cite{Quint}):
\begin{Q} For a simple real algebraic group $G$ with
$\op{rank} G\ge 2$,
is every discrete subgroup $\Ga<G$ with $\La=\La_h=\F$ 
necessarily a cocompact lattice in $G$?
\end{Q}

\subsection*{Acknowledgments} We would like to thank Dick Canary and Pratyush Sarkar for helpful conversations regarding Corollary \ref{th2}. The first named author would also like to thank Subhadip Dey and Ido Grayevsky for helpful and enjoyable discussions. We thank the anonymous referee for pointing out to us the paper \cite{Conze-Guivarch}.

\section{Preliminaries}\label{Section_preliminaries}

Let $G$ be a connected, semisimple real algebraic group.   We fix, once and for all, a Cartan involution $\theta$ of the Lie algebra $\mathfrak{g}$ of $G$, and decompose $\fg$ as $\frak g=\frak k\oplus\mathfrak{p}$, where $\fk$ and $\fp$ are the $+ 1$ and $-1$ eigenspaces of $\theta$, respectively. 
	We denote by $K$ the maximal compact subgroup of $G$ with Lie algebra $\fk$,

	Choose a maximal abelian subalgebra $\fa$ of $\frak p$. Choosing a closed positive Weyl chamber $\fa^+$ of $\fa$, let $A:=\exp \frak a$ and $A^+=\exp \frak a^+$. 
	The centralizer of $A$ in $K$ is denoted by $M$, and we set 
	$N$ to be the maximal contracting horospherical subgroup: for  $a\in \inte A^+$,
	$$N =\{g\in G: a^{-n} g a^n\to e\text{ as $n\to +\infty$}\}.$$
	We set $P=MAN$, which is the unique minimal parabolic subgroup of $G$, up to conjugation.
	
		For $u\in \fa$, we write $a_u=\exp u\in A$.
	We denote by $\|\cdot\|$ the norm on $\frak g$ induced by the Killing form.
	Consider the Riemannian symmetric space $X:=G/K$ with the metric induced from the norm $\|\cdot \|$ on $\frak g$ and $o=K\in X$.
	
Let  $\Furst=G/P$ denote the Furstenberg boundary.
	Since $K$ acts transitively on $ \Furst $ and $K\cap P=M$, we may identify $\Furst=K/M$. We denote by $\F^{(2)}$ the unique open $G$-orbit in $\Furst \times \Furst$.

	Denote by $w_0\in K$  the unique element in the Weyl group
	such that $\op{Ad}_{w_0}\fa^+= -\fa^+$; it is the longest Weyl element. We then have  $\check{P}:=w_0 P w_0^{-1}$ is an opposite parabolic subgroup of $G$, with $ \check{N} $ its unipotent radical.
 The map $\op{i}=-\op{Ad}_{w_0}: \fa^+\to \fa^+$ is called
	the opposition involution.

	For $g\in G$, we consider the following visual maps
	$$g^+:=gP\in \F\quad\text{ and }\quad g^-:=gw_0P\in \F.$$
	Then $\F^{(2)}=\{(g^+, g^-)\in \F\times \F: g\in G\}$.
	
	Any element $ g \in G $ can be uniquely decomposed as the commuting product $ g_h,g_e,g_u $, where $ g_h $, $ g_e $, and $ g_u $  are hyperbolic, elliptic and unipotent elements respectively. The Jordan projection of $g$ is defined as the element $\lambda(g) \in \fa^+ $ satisfying $ g_h= \varphi \exp \lambda(g)  \varphi^{-1} $ for some $\varphi\in G$. 
	
	An element $g\in G$ is called loxodromic if $ \lambda(g) \in \inte \fa^+ $; in this case, $g_u$ is necessarily trivial.
For a loxodromic element $g\in G$, the point $\varphi^+ \in \Furst$ is called the attracting fixed point of $g$, which we denote by $y_g$.
	For any loxodromic element $g\in G$ and $\xi\in \F$
	with $(\xi, y_{g^{-1}})\in \F^{(2)}$, we have
	$\lim_{k\to \infty} g^k \xi = y_{g}$ and the convergence is uniform on compact subsets.
	
	Note that for any loxodromic element $g\in G$,
	$$\lambda(g^{-1})= \op{i}\lambda(g).$$
	Let $\Ga<G$ be a Zariski dense discrete subgroup of $G$.
	The limit cone $\cal L=\L_\Ga$ of $\Ga$ is the smallest closed cone of $\fa^+$ containing $\lambda(\Ga)$. It is a convex cone with non-empty interior \cite{Benoist}.
	
	We will use the following simple lemma.
	\begin{lemma}\label{lemma_attracting_fixed_pt}
	For any $ v \in \lambda(\Gamma) $ and $ \zeta \in \Furst $, there exists a loxodromic element $ \gamma \in \Gamma $ with $ \lambda(\gamma)=v $ and a neighborhood $ U $ of $ \zeta $ in $ \Furst $ such that
	$ \{y_\ga \}\times U $ is a relatively compact subset of
	$\Furst^{(2)} $ and as $k\to \infty$,
		\[ \gamma^{-k} \zeta \to y_{\ga^{-1}} \quad \text{uniformly on $U$}. \]
			\end{lemma}
		\begin{proof}
Let $\zeta\in \F$.	Choose $\ga_1\in \Ga$ such that $\lambda(\ga_1)=v$. 
Since the set of all loxodromic elements of $\G$ is Zariski dense in $G$ \cite{Benoist_II} and $\F^{(2)}$ is Zariski open in $\F\times \F$,
there exists $\ga_2\in \Ga$ such that $(\zeta, \ga_2 y_{\ga_1})\in \F^{(2)}$. Let $\ga=\ga_2 \ga_1\ga_2^{-1}$, so that
$y_{\ga}=\ga_2y_{\ga_1}$. It now suffices to take any  neighborhood
$U$ of $\zeta$ such that $U\times\{ \ga_2y_{\ga_1}\}$ is a relatively compact subset of  $\F^{(2)}$.\end{proof}

\subsection*{Convergence of a sequence in $X$ to $\F$}
By the Cartan decomposition $G=KA^+K$,
for $g\in G$, we  may write 
$$g=\kappa_1(g)\exp (\mu(g))\kappa_2(g)\in KA^+K$$ where $\mu(g)\in \fa^+$, called the Cartan projection of $g$, is uniquely determined, and $\kappa_1(g), \kappa_2(g)\in K$. If $\mu(g)\in \inte \fa^+$, then $[\kappa_1(g)]\in K/M=\cal F$ is uniquely determined.

Let $\Pi$ be the set of simple roots for $(\frak g, \frak a)$. For a sequence $g_i\to G$, we say $g_i\to \infty$ regularly
if $\alpha(\mu (g_i))\to \infty$ for all $\alpha\in \Pi$. Note that if $g_i\to \infty$ regularly,
then for all sufficiently large $i$,
$\mu(g_i)\in \inte\fa^+$  and hence $[\kappa_1(g_i)]$ is well-defined.

\begin{Def}\label{conver} 
 A sequence $p_i \in X$ is said to converge to $\xi\in \cal F$ if 
 there exists $g_i\to \infty$ regularly in $G$ with $p_i=g_i(o)$
 and $\lim_{i\to\infty}[\kappa_1(g_i)]=\xi$.
\end{Def}

\subsection*{$P^\circ$-minimal subsets} 
We denote by $\La\subset \F$ the limit set of $\G$, which is defined as
\be\label{ls} \La=\{\lim \ga_i(o): \ga_i\in \Ga\}.\ee 
For a non-Zariski dense subgroup,
$\La$ may be an empty set. For $\Ga<G$ Zariski dense, this is the unique $\Ga$-minimal subset of $\F$ (\cite{Benoist}, \cite{LO1}).

It follows that  the following set $\E$ is the unique $P$-minimal subset of $\Ga\ba G$:
	$$\E=\{[g]\in \Ga\ba G: g^+\in \La\}.$$
	
	Let $P^\circ$ denote the identity component of $P$.
	Then $\E$ is a disjoint union of at most $[P: \pc]$-number of $P^\circ$-minimal subsets. We fix one $\pc$-minimal subset $\E_0$ once and for all. Note that any $\pc$-minimal subset is then of the form $\E_0 m$ for some $m\in M$.
	We set
	\be\label{ommm} \Omega:=\{[g]\in \Ga\ba G: g^+, g^-\in \La\}\quad\text{and}\quad
	\Omega_0:=\Omega\cap \E_0 .\ee

\subsection*{Busemann map} The Iwasawa cocycle $\sigma: G\times \Furst \to \fa$ is defined as follows: for $(g, \xi)\in G\times \Furst $ with $\xi=[k]$ for $k\in K$,
		$\exp \sigma(g,\xi)$ is the $A$-component of $g k$ in the $KAN$ decomposition, that is,
		$$gk\in K \exp (\sigma(g, \xi)) N.$$
		The $\fa$-valued Busemann function $\beta: \Furst \times X \times X \to\fa $ is defined as follows: for $\xi\in \Furst $ and $g, h\in G$,
			$$\beta_\xi ( ho, go):=\sigma (h^{-1}, \xi)-\sigma(g^{-1}, \xi).$$
We note that for any $g\in G$, $\xi\in \F$, and $x,y, z\in X$,
\be\label{bus} \beta_\xi(x,y)=\beta_{g\xi}(gx, gy),\quad\text{and}\quad  \beta_\xi(x, y)=\beta_\xi(x,z)+\beta_\xi(z, y).\ee 
\noindent In particular, $ \beta_\xi(o,go) \in \fa $ is defined by
\begin{equation}\label{eq_Busemann}
	g^{-1}k_\xi \in K \exp(-\beta_\xi(o, go)) N,
\end{equation}
and hence $\beta_P (o, a_u o)=u$ for any $u\in \fa$.
For $h, g\in G$, we set $\beta_\xi(h, g):=\beta_{\xi}(ho, go)$.

\medskip

\subsection*{Shadows} For $ q\in X$ and $r>0$, we set $B(q,r)=\{x\in X: d(x, q)\le r\}$.
For $p=g(o)\in X$,
the shadow of the ball $B(q,r)$ viewed from $p$ is defined as
 $$O_r(p,q):=\{(gk)^+\in \cal F: k\in K,\;
  gk\inte A^+o\cap  B(q,r)\ne \emptyset\}.$$
Similarly, for $\xi\in \F$, the shadow of the ball $B(q,r)$ as viewed from $\xi$ is
$$O_r(\xi,q):=\{h^+\in \cal F: h\in G\text{ satisfies } h^-=\xi,\, ho\in  B(q,r)  \}.$$

\begin{lem} \cite[Lemma 5.6 and 5.7]{LO1} \label{sh}
\begin{enumerate}
    \item 
There exists $\kappa>0$ such that for any $g\in G$ and $r>0$,
$$\sup_{\xi \in O_r(g(o),o)}\|\beta_{\xi}(g(o),o)-\mu(g^{-1})\|\le \kappa r .$$ 

\item If a sequence $p_i\in X$ converges to $\xi\in \cal F$, then for any $0<\e<r$,
we have 
$$O_{r-\epsilon} (p_i, o)\subset O_r(\xi, o)\subset 
O_{r+\epsilon} (p_i, o)$$ for all sufficiently large $i$.

\end{enumerate}
\end{lem}

\section{Horospherical limit points}
Let $\Ga<G$ be a Zariski dense discrete subgroup.
A $\Gamma$-tight horoball based at $\xi\in \F$ is a subset of the form $\cal H_{\xi}=g N (\exp \cal C)  (o)$ where $g\in G$ is such that $\xi=gP$ and $\cal C$ is a closed cone contained in $\inte \L\cup\{0\}$. For $T>0$, we write $\cal H_{\xi}(T)= gN (\exp (\cal C-\cal C_T))o$. We recall the definition from the introduction:
\begin{Def} \label{dhoro} We say that $\xi\in \cal F$  is a horospherical limit point of $\Ga$ if 
   there exists a $\Ga$-tight horoball $\cal H_\xi$ based at $\xi$ such that  $\cal H_\xi(T)\cap \Ga(o)\ne \emptyset $ for all $T>1$.
 \end{Def}
In this section we provide a mostly self-contained proof of the following theorem:
\begin{thm}\label{NMhoro} Let $[g]\in \E$. The following are equivalent:
\begin{enumerate}
    \item  $g^+=gP\in \Lambda$ is a horospherical limit point;
    \item $[g]NM$ is dense in $\E$.
\end{enumerate}
\end{thm}

The main external ingredient in our proof is the density of the group generated by the Jordan projection $ \lambda(\Gamma)$, due to Benoist \cite{Benoist_II}, that is,
\[ \fa = \overline{\langle \lambda(\Gamma) \rangle}  \]
for every Zariski dense discrete subgroup $ \Gamma < G $. In fact, for every  cone $ \cal C \subset \L $ with non-empty interior, there exists a Zariski dense subgroup $ \Gamma' < \Gamma $ with $ \L_{\Gamma'} \subset \cal C $ (see \cite{Benoist}); therefore we have
\begin{equation*}
	\fa = \overline{\langle \lambda(\Gamma) \cap \inte \L \rangle}.
\end{equation*}
\medskip

It is convenient to use a characterization of horospherical limit points in terms of the Busemann function.
\begin{lem}\label{hd} 
For $ \xi \in \La $, we have $\xi\in \La_h$  if and only if there exists a closed cone $\cal C \subset {\inte \L} \cup \{0\}$ and a sequence $ \gamma_j \in \Gamma $ satisfying
	\begin{equation}\label{h1}  \beta_\xi(o,\gamma_jo) \to \infty \quad \text{and } \quad \beta_\xi(o,\gamma_jo) \in \cal C \text{ for all large } j\ge 1.  \end{equation} 
\end{lem}
\begin{proof} 
Let $\xi=gP\in \La_h$ be as defined in Definition \ref{dhoro}. Then there exists $\gamma_j= gp n_j a_{u_j}k_j\in \Gamma$
for some $p\in P$, $n_j\in N$, $k_j\in K$ and $u_j\to \infty$ in some closed cone $\cal C$ contained in $\inte\L\cup\{0\}$. Fix some closed cone $\cal C'\subset \inte \L\cup\{0\}$ whose interior contains $\cal C$.
Note that
\begin{align*}
    \beta_{\xi}(o, \ga_j o) &= \beta_{gP}(e, g)+ \beta_{gP}(g, gpn_j a_{u_j})\\ &=
\beta_P(g^{-1}, e)+ \beta_P(e,p)+ \beta_P(e, n_j)+\beta_P(e,  a_{u_j})\\
&
=\beta_{P}(g^{-1},p)+u_j.
\end{align*}
Therefore the sequence $\beta_{\xi}(o, \ga_j)-u_j$ is uniformly bounded.   Since $u_j\in \cal C$,
$\beta_{\xi}(o,\ga_j o)\in \cal C'$ for all large $j$. Therefore \eqref{h1} holds. For the other direction,
let $\ga_j$ and $\cal C$ satisfy \eqref{h1} for $\xi=gP$ for $g\in G$.
 Since $G=gNAK$, we may write
$\ga_j= gn_ja_{u_j} k_j $ for some $n_j\in N, u_j\in \fa$ and $k_j\in K$. By a similar computation as above,
the sequence $\beta_\xi(o, \ga_j o)-u_j$ is uniformly bounded.
It follows that $u_j\in \cal C'$ for all large $j$ and $u_j\to \infty$. Therefore for any $T>1$, there exists $j>1$ such that
$\ga_j (o)\in g N \exp (\cal C'-\cal C'_T)(o)$. This proves $\xi\in \La_h$.
\end{proof}

We note that condition \eqref{h1} is independent of the choice of basepoint $ o $. Indeed, for any $g\in G$ and $\xi\in \F$ and for all $ \gamma \in \Gamma $ we have
\[ \beta_\xi(o, \ga o) = \beta_\xi(o, go)+\beta_\xi (go, \ga go)+\beta_\xi (\ga go, \ga o) ,\]
and hence
\begin{align*}
	\|\beta_\xi(o, \ga o)-\beta_\xi (go, \ga go)\|&=\|\beta_\xi(o, go)+\beta_\xi (\ga go, \ga o)\|\\
	&=\|\beta_\xi(o, go)-\beta_{\gamma^{-1}\xi} (o,go)\| \\ &
 \le 2 \cdot \max_{\eta \in \Furst} \|\beta_\eta (o, go)\| .
\end{align*}
Since this bound is independent of $\ga\in \Ga$,
condition \eqref{h1} implies that for any $p=go \in X$,
\begin{equation}\label{h2}  \beta_\xi(p,\gamma_j p) \to \infty \quad \text{and } \quad \beta_\xi(p,\gamma_jp) \in \cal C \text{ for all large }j.  \end{equation} 
\medskip

Let us now consider the following seemingly stronger condition for a limit point being horospherical:
\begin{definition} For $u\in \fa^+$,
	 a point $ \xi \in \cal F$ is called a $ u $-horospherical limit point if
	 for some $p\in X$ (and hence for any $p\in X$),
	 there exists a constant $ R>0 $ and a sequence $ \gamma_j \in \Gamma $ satisfying
		\[ \beta_\xi(p,\gamma_jp)\to \infty \quad \text{and}\quad \|\beta_\xi(p,\gamma_jp)-\R_+u\| < R \quad \text{for all }j. \]
	We denote the set of $ u $-horospherical limit points by $ \Lambda_h(u) $.
\end{definition}

By $ G $-invariance of the Busemann map, the set of horospherical (resp.
$u$-horospherical) limit points is $\Ga$-invariant. Therefore for $x=[g]\in \Ga\ba G$,
we may say $x^+:=\Ga gP$ horospherical (resp.
$u$-horospherical) if $g^+$ is.

For $u\in \fa$, we call $x\in \G\ba G$ a $u$-periodic point  if
$x a_u = xm_0$ for some $m_0\in M$; note that $xa_{\br u} M_0$ is then compact. Note that for $u\in \inte\fa^+$, the existence of a $u$-periodic point is equivalent to the condition that $u\in \lambda(\Ga)$.

\begin{lem} \label{hor}  Let $u\in \fa^+$.
If $x\in \Ga\ba G$ is $u$-periodic, then $x^+\in \F$ is a
$u$-horospherical limit point.
\end{lem}
\begin{proof} Since $x$ is $u$-periodic, there exist
$g\in G$ with $x=[g]$ and $\ga\in \Ga$ such that
$\ga=g a_u m g^{-1}$ for some $m\in M$, and $y_\ga=g^+ \in \La$.
Moreover, for any $k\ge 1$
$$\beta_{gP} (go, \ga^k go)=\beta_{P}(o, a^k_u o)=k u.$$
This implies $gP$ is $u$-horospherical.
\end{proof}

\begin{prop}\label{prop_tau_horo_to_periodic}
Let $x\in \Ga\ba G$.
	If $ x^+$ is $u$-horospherical for some $u\in \lambda(\Gamma)$ then
	the closure $\overline{xN}$ contains a $ u $-periodic point.
\end{prop}

\begin{proof} Choose $g\in G$ so that $x=[g]$.
	We may assume without loss of generality that $g=k\in K$, since
	$kan N= k N a$, and a translate of a $u$-periodic point by an element of $A$ is again a $u$-periodic point.
	Since $u\in \lambda(\Gamma)$,
	there exists a $u$-periodic point, say, $x_0\in \Ga\ba G$.
	It suffices to show that
	\begin{equation}\label{suf} \overline{[k]N}\cap x_0 AM\ne \emptyset\end{equation}
	as every point in $x_0AM$ is $u$-periodic.
	
	Since $k^+$ is $ u $-horospherical and using \eqref{eq_Busemann}, there exists $ R>0 $ and sequences $ \gamma_j \in \Gamma $, $ u_j \to \infty $ in $ \fa^+ $ and $ k_j \in K $ and $ n_j \in N $ satisfying
	$\gamma_j^{-1} k = k_j a_{-u_j} n_j $
	or
	\begin{equation}\label{eq_Gamma_k_as_k_xiAMN}
		k_j = \gamma_j^{-1} k n_j^{-1} a_{u_j},
	\end{equation} 
	with $ \|\R_+u - u_j \| < R $ for all $ j $. Let $ \ell_j \to \infty $ be a sequence of integers satisfying \begin{equation}
	   \label{lj}  \|\ell_ju - u_j \| < R+\|u\| \;\quad\text{for all $ j \ge 1$}.	\end{equation}
	
By passing to a subsequence, we may	assume without loss of generality that $ \gamma_j^{-1}kP$ converges to
some $\xi_0 \in \Furst $. Since $\check{N}P$ is Zariski open and $ \Gamma $ is Zariski dense,
we may choose $g_0\in G$ such that $x_0=[g_0]$ and  $g_0^{-1}\xi_0\in \check{N}P$. Let $h_0\in \check{N}$ be such that
 $ \xi_0=g_0 h_0 P  $.
		\newcommand{\tn}{\tilde{n}}
	Since $ g_0\check{N}P $ is open and $\ga_j^{-1}kP\to g_0h_0P$,
	we may assume that for all $j$, there exists $ h_j \in \check{N} $ 
satisfying $ g_0 h_j P = \gamma_j^{-1}k P = k_j P $ with $ h_j \to h_0 $. Let $p_j= a_{v_j}m_j \tn_j \in P=AMN$ be such that $g_0 h_j p_j= k_j$; since $h_j\to h_0$ and the product map $\check{N}\times P\to \check{N}P$ is a diffeomorphism, the sequence $p_j$, as well as ${v_j}\in \fa$, are  bounded.

Therefore by \eqref{eq_Gamma_k_as_k_xiAMN}, we get for all $j$,
	\begin{align*}
		g_0 &= k_j p_j^{-1}h_j^{-1}
		\\& =\gamma_j^{-1} kn^{-1}_j a_{u_j} ( \tn^{-1}_j m_j^{-1} a_{-v_j}) h_j^{-1}  \\
		&= \gamma_j^{-1} k n^{-1}_j (a_{u_j} \tn^{-1}_j a_{-u_j}) a_{u_j} m_j^{-1} a_{-v_j}h_j^{-1}  \\
		&= \gamma_j^{-1} k n^{-1}_j (a_{u_j} \tn^{-1}_j a_{-u_j}) m_j^{-1} (a_{u_j-v_j}h_j^{-1} a_{-u_j+v_j})  a_{u_j-v_j}.
	\end{align*}

Since $h_j^{-1} \in \check{N}$ and $ v_j \in \fa $ are uniformly bounded and since
$u_j\to \infty$ within a bounded neighborhood of the ray $\br_+ u\in \inte\fa^+$, we have
	\[ {\tilde h}_j = a_{u_j-v_j}h_j^{-1}  a_{-u_j+v_j} \to e \quad \text{in }\check{N}. \]
By setting $ n'_j = n^{-1}_j (a_{u_j} \tn^{-1}_j a_{-u_j}) \in N $, 	we may now write
	\[ g_0 = \gamma_j^{-1} k n'_j m_j^{-1} {\tilde h}_j a_{u_j-v_j}. \]
	
Since $x_0$ is $u$-periodic, there exists $ \gamma_0 \in \Gamma $ such that  $ \gamma_0 = g_0 a_u m_0 g_0^{-1} $ for some $ m_0 \in M $. Hence for all $j\ge 1$,
	\begin{align*}
		\gamma_0^{-\ell_j} &= g_0 a_{-\ell_j u} m_0^{-\ell_j} g_0^{-1} = (\gamma_j^{-1} kn'_j  m_j^{-1} {\tilde h}_j   a_{u_j-v_j} ) ( a_{-\ell_j u} m_0^{-\ell_j }) g_0^{-1}.
	\end{align*}
In other words,
	\[ \gamma_j^{-1} k n'_j = \gamma_0^{-\ell_j}g_0 m_0^{\ell_j} a_{-u_j+\ell_j u+v_j} {\tilde h}_j^{-1}m_j. \]
	Since the sequence $ -u_j+\ell_j u+ v_j \in \fa $ is uniformly bounded by \eqref{lj} and  $ {\tilde h}_j \to e $ in $ \check{N} $, we conclude that the sequence $ \Gamma k n'_j $ has an accumulation point in $ \Gamma g_0 AM $. This proves \eqref{suf}.
\end{proof}

It turns out that a horospherical limit point is
also $ u $-horospherical for any $u\in \inte \L$:
\begin{prop}\label{prop_horo_to_tau_horo}\label{allu}
For each $ u \in {\inte \L} $, we have $ \Lambda_h=\Lambda_h(u) $. 
\end{prop}

\begin{proof}
 Let $ \xi \in \La_h $. By definition, there is a sequence
 $ \gamma_j \in \Gamma $  satisfying 
 $v_j:=\beta_\xi(e,\gamma_j)\to \infty$ with the sequence
 $\|v_j\|^{-1}v_j$ converging to some point $v_0\in \inte \L$. By passing to a subsequence, we may assume that $\ga_j^{-1} \xi$ converges to some $\xi_0\in \F$.
 
 Let $u\in \inte\L$. 
 We claim that $\xi\in \La_h(u)$.
 We first consider the case $u\not\in \br_+ v_0$. Let $r:=\op{rank} G-1\ge 0$.
 Since $\cup_{\ga\in \Ga}\br_+\lambda(\ga)$ is dense
 in $\L$, there exist $w_1, \cdots, w_r\in \lambda(\Gamma)$ such that $v_0$ belongs to the interior of the convex cone spanned by $u, w_1, \cdots, w_r$, so that
 $$v_0= c_0 u+ \sum_{\ell=1}^r c_\ell w_\ell $$
for some positive constants  $c_0,\cdots, c_\ell$.

Since $\|v_j\|^{-1}v_j \to v_0$, we may assume, by passing to a subsequence, that for each $j\ge 1$, we have
\be\label{uuu} \|v_j\|^{-1} v_j = c_{0,j} u +\sum_{\ell =1}^r c_{\ell,j} w_\ell  \ee
for some positive $c_{\ell,j}$, $\ell=0, \cdots, r$.
  Note that for each $0\le \ell \le r$, $c_{\ell ,j}\to c_\ell $ as $j\to \infty$.
 
By \Cref{lemma_attracting_fixed_pt}, we can find a loxodromic element $g_1\in \Ga$ and a neighborhood $U_1$
of $\xi_0$ such that $\lambda(g_1^{-1})=w_1$, $\{y_{g_1}\}\times  U_1\subset \F^{(2)}$ and
$g_1^{-k}U_1\to y_{g_1^{-1}}$ uniformly. Applying \Cref{lemma_attracting_fixed_pt} once more, we can find $ g_2 \in \Gamma $ satisfying $ \lambda(g_2^{-1}) = w_2 $ and a neighborhood $ U_2 \subset \Furst $ of $ y_{g_1^{-1}} $ satisfying $ \{y_{g_2} \}\times U_2 \subset \Furst^{(2)} $ and that $ g_2^{-k}U_2 \to y_{g_2^{-1}} $ uniformly.

Continuing inductively, we get elements $ g_1,...,g_r \in \Gamma $ and  open sets $ U_1,...,U_r \subset \Furst $ satisfying that	for all $ \ell  = 1,...,r $,
	\begin{enumerate}
		\item $ {w}_\ell=\lambda(g_\ell^{-1}) $;
		\item $ y_{g_{\ell-1}^{-1}} \in U_{\ell} $; 
		\item $ g_\ell^{-k}U_\ell \to y_{g_\ell^{-1}} $ uniformly; and
		\item $ \{ y_{g_\ell} \}\times U_\ell $ is a relatively compact subset of
		$ \Furst^{(2)} $.
	\end{enumerate}

We set $\xi_\ell:=y_{g_\ell^{-1}}$ for each $1\le \ell\le r$; so $U_{\ell}$ is a neighborhood of $\xi_{\ell-1}$ for each $1\le \ell\le r$.

Since $\cal Q_{\eta_0}:=\{\eta\in \F: (\eta_0, \eta)\in \F^{(2)}\}=\bigcup_{R>0} O_R (\eta_0, o)$ for any $\eta_0\in \F$ and
 $U_\ell\subset \cal Q_{y_{g_\ell}}$
is a relatively compact subset of $\Furst^{(2)} $, there exists $ R_\ell>0 $ such that $ U_\ell \subset O_{R_\ell}(y_{g_\ell},o) $. Since $g_\ell^ko $ converges to
$y_{g_\ell}$ as $k\to +\infty$,
by Lemma \ref{sh}(2), 
\be\label{or} O_{R_\ell}(y_{g_\ell}o,o) \subset O_{R_\ell+1}(g_\ell^k o,o) \ee  for all sufficiently large $k>1$.
	\medskip
	
	For each $1\le \ell\le r$ and $j\ge 1$, let $k_{\ell,j}$ be the largest integer smaller than $c_{\ell, j}\|v_j\|$. As $\|v_j\|\to \infty$, and $c_{\ell,j}\to c_\ell$, we have $k_{\ell, j}\to \infty$ as $j\to \infty$.
	By the uniform contraction  $ g_\ell^{-k}U_i \to \xi_\ell$, 
	there exists $j_0>1$ such that for all $j\ge j_0$,  
	\be\label{gj} \gamma_{j}^{-1}\xi \in U_1, \quad  g_\ell^{-k_{\ell,j}}U_\ell \subseteq U_{\ell+1} ,\quad  \text{and} \quad
	U_\ell \subset O_{R_\ell+1}(g_\ell^{k_\ell, j}o,o) \ee
		for all $ \ell=1,...,r $.

For each $j\ge j_0$, we now set
	\[ 	\tilde{\gamma}_j:=\gamma_j g_1^{k_{1,j}} g_2^{k_{2,j}} \cdots g_r^{k_{r,j}} \in \Gamma. \]

We claim that $\beta_\xi(e,\tilde{\gamma}_j)\to \infty$ as $j\to \infty$ and that
	\begin{equation} \label{beta} \sup_{j\ge j_0}
	\left\|\beta_\xi(e,\tilde{\gamma}_j)-\br_+ u \right\|<\infty ;\end{equation}
this proves that $\xi$ is $u$-horospherical.

Fix $ j\ge j_0 $ and for each $1\le \ell\le r$, let $ k_\ell:=k_{\ell,j} $, $b_\ell :=c_{\ell,j}\|v_j\|$, and set
	\[ h_\ell = g_1^{k_1} g_2^{k_2} \cdots g_\ell^{k_\ell}, \]
	and $ g_0=e $.
	The cocycle property of the Busemann function gives that
	\begin{equation}	\beta_\xi(e,\tilde{\gamma}_j) =\beta_\xi(e,\gamma_j) -\sum_{\ell=1}^{r} \beta_{h_{\ell-1}^{-1}\gamma_j^{-1}\xi}(g^{k_\ell}_\ell, e)  .\label{eq_Busemann_vj_as_sum}
	\end{equation}

	By \eqref{gj}, $\ga_j^{-1}\xi\in U_1$ and
for each $1\le \ell \le r$,
	$$h_{\ell -1}^{-1}\ga_j^{-1} \xi \in 
	g_\ell^{-k_\ell}\cdots g_1^{-k_1}U_1\subset U_{\ell+1}\subset O_{R_\ell +1}(g_\ell^{k_\ell}o, o ).$$
	
	Hence by Lemma \ref{sh}(1), there exists $\kappa\ge 1$
	such that for each $1\le \ell \le r$
$$	\| \beta_{h_{\ell-1}^{-1}\gamma_j^{-1}\xi}(g^{k_\ell}_\ell, e) 
	-\mu(g^{-k_\ell}_\ell)\|\le \kappa (R_{\ell}+1).$$

Note that for some $C_\ell>0$, $\|\mu(g_\ell^{-k}) -k
\lambda(g_\ell^{-1})\|\le C_\ell$ for all $k\ge 1$.
Since $\lambda(g_\ell^{-1})=w_\ell$, we get

$$	\| \beta_{h_{\ell-1}^{-1}\gamma_j^{-1}\xi}(g^{k_\ell}_\ell, e) 
	-k_\ell w_\ell \|\le \kappa (R_{\ell}+1)+ C_\ell .$$
	
Therefore by \eqref{eq_Busemann_vj_as_sum},
we obtain
$$\|\beta_\xi(e, \tilde \ga_j) -(v_j -\sum_{\ell=1}^r
k_\ell w_\ell ) \| \le \kappa \sum_{\ell=1}^r  (R_\ell +  C_\ell+1) .$$

By \eqref{uuu}, we have
$$c_{0,j} \|v_j\| u =  v_j -\sum_{\ell=1}^{r}  b_\ell  w_\ell .$$
Since $|b_\ell  -k_\ell|\le 1$ and $c_{0,j}>0$,
we deduce that for all $j\ge j_0$,
\begin{align*}
&\|\beta_\xi(e, \tilde \ga_j)-\br_+ u\| \le
\left\| \beta_\xi(e, \tilde \ga_j)-  c_{0,j} \|v_j\|\cdot  u \right\| \\
&\le \left\| \beta_\xi(e, \tilde \ga_j)-  (v_j -\sum_{\ell=1}^r
k_\ell w_\ell ) \right\| +   \sum_{\ell=1}^r \left\| 
k_\ell w_\ell  -b_\ell w_\ell \right\|  
\\&\le \kappa \sum_{\ell=1}^r  (R_\ell +  C_\ell +\|w_\ell\| +1) .
\end{align*}

This proves $\eqref{beta}$, and consequently $\xi$ is $u$-horospherical for any $u\notin \br_+ v_0.$
To show that $\xi$ is $v_0$-horospherical,
fix any $u\notin \br_+v_0$ and $\tilde \ga_j\in \G$ be a sequence as in \eqref{beta} associated to $u$. If we set $\tilde v_j=\beta_\xi (e,
\tilde \ga_j)$, then $\|\tilde v_j\|^{-1} \tilde v_j$ converges to
a unit vector in $\inte \L$ proportional to $u$.
Therefore by repeating the same argument only now switching the roles of $v_0$ and $u$, we prove that $\xi$ is $v_0$-horospherical as well. This completes the proof.
\end{proof}

We may now prove \cref{NMhoro}:

\begin{proof}[Proof of \cref{NMhoro}] 
Let $g\in G$ be such that $\xi=g^+\in \La$ is a horospherical limit point.
Set $Y:=\overline{[g]NM}$. We claim that $Y=\cal E$.
By Benoist \cite{Benoist}, the group generated by
$\lambda(\Ga)\cap \inte \L$ is dense in $\fa$.
 Hence for every $ \varepsilon > 0 $ there exist loxodromic elements $ \gamma_1,...,\gamma_{q} \in \Gamma $ such that
	\[ \lambda(\gamma_1),...,\lambda(\gamma_{q}) \in \mathrm{Int}\L \]
	and the group  $\z\lambda(\gamma_1)+\cdots +\z \lambda(\gamma_{q})$ is an $\e$-net in $\fa$, i.e., its $\e$-neighborhood covers all $\fa$.	Denote $u_i = \lambda(\gamma_i) $ for $ i=1,...,q $. By \Cref{prop_horo_to_tau_horo}, the point $ \xi $ is $u_1 $-horospherical. By \Cref{prop_tau_horo_to_periodic}, there exists a $ u_1 $-periodic point $x_1 \in \E $ contained in $ Y $, set
	\[ Y_1: = \overline{x_1 NM} \subset Y. \]

By Lemma \ref{hor}, $x_1^+$ is $u_1$-horospherical; in particular, it is a horospherical limit point. Therefore
 we can inductively find
a $u_i$-periodic
	point $ x_i $ in $ Y_{i-1}=\overline{x_{i-1} NM} $ 	for each $2\le i\le q$. By periodicity $x_i (\exp {u_i}) M=x_i M$, and hence
	$Y_i \exp {\mathbb Z u_i} = Y_i$ for each $1\le i\le q$. Therefore we obtain 
	$$Y\supset Y_1\, {\exp {\z u_1}}\supset Y_2\, {\exp ({\z u_1+\z u_2})}\supset \cdots \supset Y_q\, {\exp {(\sum_{i=1}^q \z u_i})}.$$
	
	Recalling the dependence of $Y_q$ and $\sum_{i=1}^q \z u_i$ on $\e$,
	set $$Z_\e:= Y_q MN \exp (\sum_{i=1}^q \z u_i) \subset Y.$$
Since $MN \exp (\sum_{i=1}^q \z u_i)$ is an $\e$-net of $P$ and
$\E$ is $P$-minimal, $Z_\e$ is a $2\e$-net of $\E$ for all $\e>0$.
Since $Y$ contains a $2\e$-net of $\E$ for all $\e>0$ and $Y$ is closed,
it follows that $Y=\E$.
	
	For the other direction, suppose that $[g]NM$ is dense in $\E$ for $g\in G$.
 Choose any $u\in \inte \L$ and a closed cone $\cal C
 \subset \inte \L\cup \{0\}$ which contains $u$. Then $\cal H_\xi=gN(\exp \cal C )(o)$ is a $\Ga$-tight horoball. Let $t>1$.
 Since $ga_{-2t u}\in \cal E$, there exist $\ga_i\in \Ga$, $n_i\in N$, $m_i\in M$ and $q_i\to e$ in $G$ such that
 for all $i\ge 1$,
 $\ga_i g n_i m_i q_i= g a_{-2t u}$. Since $d(\ga_i^{-1}g, gn_i m_ia_{2t u})\le d(q_i a_{2t u},a_{2t u})\to 0 $
as $i\to \infty$,  it follows that for all sufficiently large $i\ge 1$, $\ga_i^{-1} go \in \cal H_\xi (t)$. Hence $g^+$ is a horospherical limit point by Definition \ref{dhoro}.
\end{proof}

\section{Topological mixing and directional limit points}

There is a close connection between denseness of $ N $-orbits and the topological mixing of one-parameter diagonal flows with direction in $ \inte \L $. This connection allows us to make use of recent topological mixing results by Chow-Sarkar \cite{CS}: recall the notation $\Omega_0$ from \eqref{ommm}.
\begin{thm}\cite{CS} \label{cs} For any $u\in \inte \L$, 
	$\{a_{tu} : t\in \br\}$ is topologically mixing on $\Omega_0$, i.e., for any open subsets $\cal O_1, \cal O_2$ of $\Ga\ba G$ intersecting $\Omega_0$,
	$$\cal O_1\exp tu \cap \cal O_2\ne \emptyset \quad\text{ for all large $|t|\gg 1$}.$$ 
\end{thm}

\noindent The above theorem was predated by a result of Dang \cite{Dang} in the case where $ M $ is abelian.

\subsection*{$N$-orbits based at directional limit points along $\inte \L$}
\begin{Def} 	For $u\in \inte \fa^+$, denote by $\La_u$ the set of all $u$-directional limit points, i.e., $\xi\in \La_u$ if and only if
	$\limsup_{t\to +\infty} \Ga g \exp (tu)\ne \emptyset$ for some (and hence any) $g\in G$
	with $gP=\xi$.
\end{Def} 
It is easy to see that $\La_u\subset \La$
for $u\in \inte \fa^+$.
\begin{prop}\label{pmix}
If $[g]\in \cal E_0$ satisfies $g^+\in \La_u$ for some $u\in \inte\L$, then $$\overline{[g]N}=\E_0.$$
\end{prop}
\begin{proof}
	Since $\Omega_0 N=\E_0$, we may assume without loss of generality that $x=[g]\in \Omega_0$.	There exist $\ga_i\in \G$ and $t_i\to+\infty $ such that $\ga_i g a_{t_iu} $ 	converges to some $h\in G$. In particular, $x \exp (t_i u)\to [h]$. Since
	$x a_{t_iu}\in \Omega_0$ and $\Omega_0$ is $ A $-invariant and closed,
	we have $[h]\in \Omega_0$.
	We write $\ga_i g a_{t_iu}=hq_i$ where $q_i\to e$ in $G$. 
	Therefore $xN=[h] q_i N a_{-t_iu}$ for all $i\ge 1$. Let $\cal O\subset \Ga\ba G$ be any open subset intersecting $\Omega_0$. It suffices to show that $xN\cap \cal O\ne \emptyset$. Let $\cal O_1$ be an open subset intersecting $\Omega_0$ and
	$V\subset \check{P}$ be an open symmetric neighborhood of $e$ such that $\cal O_1 V\subset \cal O$.
	
	Since $q_i\to e$ and $NV$ is an open neighborhood of $e$ in $G$, there exists an open neighborhood, say, $U$ of $e$ in $G$ and $i_0$ such that $U\subset q_i NV$ for all $i\ge i_0$.
	By Theorem \ref{cs}, we can choose $i>i_0$ such that
	$[h] U\cap \cal O_1 a_{t_iu}\ne \emptyset$.
	It follows that $[h] q_i NV a_{-t_iu} \cap \cal O_1\ne \emptyset $. Since
	$V\subset a_{-t_iu} V a_{t_iu}$ as $u\in \fa^+$, 
	we have
	$$[h] q_i NV a_{-t_iu} \cap \cal O_1 \subset  [h] q_i N a_{-t_iu}
	V \cap \cal O_1 .$$
	Since $V=V^{-1}$, we get
	$[h] q_i N a_{-t_iu} \cap \cal O_1 V \ne \emptyset$.
	Therefore $xN\cap \cal O\ne \emptyset$, as desired.
\end{proof}

This immediately implies:
\begin{cor}\label{cmix}  If $[g]\in \Omega_0$ is $u$-periodic for some $u\in \inte\L$, then $$\overline{[g]N}=\E_0.$$
\end{cor}
\begin{proof}
	Since $ [g] (\exp ku) =[g]m_0^k$ for any integer $k$ and $M$ is compact, we have $g^+\in \La_u$.
	Therefore the claim follows from Proposition \ref{pmix}.
\end{proof}

We may now conclude our main theorem in its fullest form:

\begin{thm}\label{mhoro} Let $[g]\in \E_0$. The following are equivalent:
	\begin{enumerate}
		\item  $g^+\in \La$ is a horospherical limit point;
		\item $[g]N$ is dense in $\E_0$;
		\item $[g]NM$ is dense in $\E$.
	\end{enumerate}
\end{thm}

\begin{proof}
The implication $(2)\Rightarrow (3)$ is trivial and $ (3) \Rightarrow (1) $ was shown in \Cref{NMhoro}. Hence let us prove $ (1)\Rightarrow (2) $.

Let $x=[g]\in \E_0$. Suppose that $g^+\in \Lambda_h$. Fix any $u\in \lambda(\Gamma)\cap \inte \L_\Ga$. By
Propositions \ref{prop_horo_to_tau_horo} and \ref{prop_tau_horo_to_periodic}, $xN$ contains a $u$-periodic point, say, $x_0$.  Hence by Corollary \ref{cmix},  $\overline{xN}\supset \overline{x_0N}\supset  \Omega_0 N=\E_0$. This proves
$(1)\Rightarrow (2)$. 
\end{proof}

\section{Conical limit points, Minimality and Jordan projection}

A point $\xi\in \cal F$ is called a \emph{conical} limit point of $\Ga$ if there exists a sequence $ u_j \to \infty $ in $ \fa^+ $ such that for some (and hence every) $ g \in G $ with $ \xi = gP $ 
\begin{equation*}
		\limsup_{j\to \infty} \Gamma g a_{u_j} \neq \emptyset.
\end{equation*}
A conical limit point of $\Ga$ is indeed contained in $\La$.
We consider the following restricted notion:
\begin{definition}
	We call $\xi\in \F$ a \emph{strongly conical} limit point of $\Ga$ if there exists a closed cone $ \cal C \subset \inte \L \cup \{0\} $ and a sequence $ u_j \to \infty $ in $ \cal C $ such that for some (and hence every) $ g \in G $ with $ \xi = gP $,
	\begin{equation*}
		\limsup_{j\to \infty} \Gamma g a_{u_j} \neq \emptyset . \end{equation*}
\end{definition}

\begin{Rmk}
   We mention that a conical limit point  defined in \cite{Conze-Guivarch}  for $\Ga<\op{SL}_d(\br)$ coincides with our strongly conical limit point.
\end{Rmk}

\begin{lem}\label{st}
Any strongly conical limit point of $\Ga$ is horospherical.
\end{lem}

\begin{proof}
Suppose that $ \xi = gP $ is strongly conical, that is, there exist
	$\ga_j\in \Ga$ and  $ u_j \to \infty $ in some closed cone $\cal C\subset \inte \L \cup\{0\}$ such that
	$\ga_j g a_{u_j}$ converges to some $h\in G$.
	Write $\ga_j g a_{u_j}=hq_j$ where $q_j\to e$ in $G$.
	Let $\cal C'$ be a closed cone contained in $\inte \L \cup\{0\}$
	whose interior contains $\cal C \smallsetminus \{0\}$.
	
	Then $\ga_j^{-1}=g a_{u_j}q_j^{-1} h^{-1}$ and
	$$\beta_{gP} (e, \ga_j^{-1} ) =\beta_P (g^{-1}, a_{u_j}q_j^{-1} h^{-1})=
	\beta_{P} (g^{-1}, q_j^{-1}h^{-1}) +\beta_P(e, a_{u_j}) .$$
	Since $\beta_P(e, a_{u_j})=u_j$ and $q_j^{-1}h^{-1}$ are uniformly bounded,
	the sequence
	$$\beta_{gP}(e, \ga_j^{-1}) - u_j$$
is	uniformly bounded. Since $u_j\in \cal C$ and $\cal C\subset \inte\cal C'\cup\{0\}$, it follows that $$\beta_{gP}(e, \ga_j^{-1})\in \cal C'$$
	for all sufficiently large $j$. This proves that $\xi\in \La_h$.

\end{proof}

\begin{cor} For any $g\in G$ with strongly conical $g^+\in \cal F$, we have
$$\overline{[g] NM}=\cal E .$$
\end{cor}

\subsection*{Directionally conical limit points}
If $v\in {\inte \L}$, then clearly $\La_v$ is contained in the horospherical limit set of $\Ga$, and hence any $NM$-orbit based at a point of $\La_v$ is dense in $\cal E$. 
On the other hand, we would like to show in this section that the existence
of a point in $\La_v$ for $v\in \partial \L_\Ga$ implies the existence of a nondense $NM$-orbit in $\cal E$.

The flow $\exp (\br u)$ is said to be topologically transitive on $\Omega/M=\{\Gamma g M : g^{\pm}\in \La\}$ if, for any
open subsets $\cal O_1, \cal O_2$ intersecting $\Omega/M$, there exists a sequence 
$t_n \to +\infty$ such that $\cal O_1\cap \cal O_2 a_{t_n u}\ne \emptyset$.

	We make the following simple observation:
	\begin{lem} For $g\in \Omega$, we have
$$\text{$\overline{gNM}\supset \Omega$\;\; 
if and only if \;\; $\overline {gw_0\check{N}M}\supset \Omega$.}$$
\end{lem}
\begin{proof}
We have $\check{N}=w_0 N w_0^{-1}$.
Note that $[g]\in \Om$ if and only if $[gw_0]\in \Om$, since $(g w_0)^{\pm}=g^{\mp}$. So $\Om w_0=\Om$.
Hence
$gNM$ is dense in $\Om$ if and only if $gw_0 \check{N}M w_0^{-1} $
is dense in $\Om$ if and only if
$[g]w_0\check{N}M$ is dense in $\Om w_0=\Om$.
\end{proof}

Since the opposition involution preserves $\L$ and $\lambda(g^{-1})=\op{i}\lambda(g)$ for any loxodromic element, it follows that
	$\lambda(\ga)\in \partial \L$ if and only if $\lambda(\ga^{-1})\in \partial \L$.

\begin{prop}\label{lah}\;
  \begin{enumerate}
      \item 
  If $\La=\La_h$, then
      $\exp (\br v)$ is topologically transitive on
     $\Omega/M$ for any $v\in \inte\fa^+$ such that
     $\La_v\ne \emptyset$.
     \item 
    For any loxodromic element $\ga \in \Ga$ with $\{y_{\ga}, y_{\ga^{-1}}\}\subset \La_h$,
     the flow $\exp (\br \lambda(\ga))$ is topologically transitive on
     $\Omega/M$.   \end{enumerate}  
\end{prop}

\begin{proof}
Assume that $\La=\La_h$; so
  the $NM$-action on $\cal E$ is minimal.    Suppose that $ \Lambda_v\ne \emptyset$ for some $v\in \inte\fa^+.$  
   We claim that for any  ${\cal O}_1, {\cal O}_2$ be two right $M$-invariant open subsets intersecting $\Omega$, 
 ${\cal O}_1 \exp (t_iv) \cap {\cal O}_2\ne \emptyset$ for some sequence $t_i \to +\infty$. 
Choose $x=[g]\in\Omega $ so that $g^+\in \Lambda_v$.
Then there exists $\ga_i\in \Ga$ and $t_i\to +\infty$ such that
$\gamma_i g a_{t_i v}$ converges to some $g_0$. Note that $x_0:=[g_0]\in \Om$.
So write $\gamma_i g a_{t_i v}= g_0 h_i$ with $h_i\to e$.
By the $NM$-minimality assumption, $x NM $ intersects every open subset of $\Om$. Since $v\in \inte \fa^+$ and hence
$a_{-tv} n a_{tv}\to e$ as $t\to +\infty$, 
we may assume without loss of generality that $x \in {\cal O}_1$.
Choose an open neighborhood $U$ of $e$ in $G$ so
that ${\cal O}_1\supset x U M$.
Note that there exists a sequence $T_i\to \infty$ as $i\to \infty$ such that for all $i$,
$$x U M a_{t_i v} \supset x a_{t_i v} a_{-t_i v} \check{N}_\epsilon M a_{t_i v} \supset 
x_0 h_i \check{N}_{T_i}, $$
where $ \check{N}_{R}=\check{N} \cap B_R^G $ is the the set of elements of $ \check{N} $ of norm$ \leq R $.
So
${\cal O}_1 a_{t_i v}\supset x_0h_i \check{N}_{T_i}$.

Choose an open neighborhood $V$ of $e$ in $G$ and some open subset ${\cal O}_2'$ intersecting $\Om$ so that ${\cal O}_2\supset {\cal O}_2' V$. 
Since $x_0\check{N}M$ is dense in $\Om$, $x_0 n\in {\cal O}_2'$ for some $n\in \check{N}$.
Hence $x_0 h_i n = x_0 n (n^{-1} h_i n)\in {\cal O}_2' V\subset {\cal O}_2 $ for all $i$ large enough so that $n^{-1} h_i n\in V$.
Therefore for all $i$ such that $n\in \check{N}_{T_i}$,
we get $$x_0h_in\in {\cal O}_1 a_{t_i v}\cap {\cal O}_2\ne \emptyset.$$ This proves the first claim.

Now suppose that $\ga\in \Ga$ is a loxodromic element with $y_{\ga}, y_{\ga^{-1}}\in \La_h$.
Write $\gamma =g m a_{ v} g^{-1} $ for some $g\in G$ and $m\in M$. Since $y_\ga=g^+$ and $y_{\ga^{-1}}=gw_0^+,$ we have
 each 
$[g]NM$
and $[g]w_0 NM$ contains $\Omega$ in its closure.
Now in the notation of the proof of the first claim, note that $x_0=[g_0]\in [g]M$
since $[g]\exp (\br v) M$ is closed. Therefore each $\overline{x_0 NM}$ and
$\overline{x_0 \check{N}M}$ contains
$ \Omega$. Based on this, the same argument as above shows the topological transitivity of $\exp \br v$, which finishes the proof since $v=\lambda(\ga)$.
\end{proof}

Since $\L$ is invariant under the opposition involution $\i$ and
$\la(\ga)=\i \la(\ga^{-1})$
for any loxodromic element $\ga\in \Ga$, the Jordan projection $\lambda(\ga)$ belongs to $ \partial \L$ if and only if the Jordan projection $\lambda(\ga^{-1})$ belongs to $ \partial \L$. 
Together with the result of Dang and Gloriuex \cite[Proposition 4.7]{DG} which say that
 $\exp (\br u)$ is not topologically transitive on $\Omega/M$ for any $u\in \partial \L\cap \inte \fa^+$,
 Proposition \ref{lah} implies the following: 
\begin{cor}\;
\begin{enumerate}[leftmargin=*]
 \item If $\La_v\ne \emptyset$ for some $v\in \partial \L\cap \inte \fa^+$, then $$\La\ne \La_h.$$
   	\item For any loxodromic element $\ga\in \Ga$, we have $\lambda(\gamma)\in \partial \L$ if and only if  $$\{y_\ga, y_{\gamma^{-1}}\} \not\subset \La_h .$$
   	Hence, if $\La=\La_h$, then $\lambda(\Ga)\subset \inte \L$.
   	\end{enumerate} 
\end{cor}

\end{document}